\documentclass{amsart}
\usepackage{url}

 \usepackage{xcolor}

\usepackage{amsmath,amsfonts,amsthm,amssymb}

\newtheorem{theorem}{Theorem}
\newtheorem{proposition}[theorem]{Proposition}

\newtheorem{remark}[theorem]{Remark}
\numberwithin{equation}{section} \numberwithin{theorem}{section}
\newcommand{\R}{\mathbb{R}}
\newcommand{\N}{\mathbb{N}}

\newcommand{\C}{\mathbb{C}}
\newcommand{\Z}{\mathbb{Z}}
\newcommand{\bP}{\mathbb{P}}
\renewcommand{\l}{\lambda}

\newcommand{\ep}{\varepsilon} 
\newcommand{\al}{\alpha} 
\newcommand{\Ea}{E_{\al}} 
\newcommand{\ua}{u_{\al}}

\begin{document}
\title[Gap theorem for $\al$-harmonic maps]{A gap theorem 
for $\al$-harmonic maps \\ between two-spheres}

\author{Tobias Lamm}
\address[T.~Lamm]{Institute for Analysis\\ 
Karlsruhe Institute of Technology (KIT)\\ 
Englerstr. 2\\ 76131 Karlsruhe\\ Germany}
\email{tobias.lamm@kit.edu}
\author{Andrea Malchiodi}
\address[A.~Malchiodi]{Scuola Normale Superiore\\ 
Piazza dei Cavalieri 7\\ 
50126 Pisa\\ Italy}
\email{andrea.malchiodi@sns.it}
\author{Mario Micallef}
\address[M.~Micallef]{Mathematics Institute\\ 
University of Warwick\\ 
Coventry CV4 7AL\\ UK}
\email{M.J.Micallef@warwick.ac.uk}

\date{\today}

\subjclass[2000]{}

\begin{abstract} 

In this paper we consider approximations introduced by Sacks-Uhlenbeck 
of the harmonic energy for maps from $S^2$ into $S^2$. 
We continue the analysis in \cite{lmm} about limits of 
$\al$-harmonic maps with uniformly bounded energy. Using a recent 
energy identity in \cite{lizhu}, we obtain an optimal gap theorem for the 
$\al$-harmonic maps of degree $-1, 0$ or $1$. 

\end{abstract}

\maketitle
\section{Introduction}

Given a compact Riemannian manifold $(N^n,h)$, 
the {\em energy functional} on the space of 
smooth closed curves $\gamma \colon S^1 \to N$ is defined by 
\begin{equation}\label{eq:1d-energy}
 E(\gamma) :=  \frac 12 \int_0^1 h(\dot{\gamma},\dot{\gamma})(t) \, dt \, . 
\end{equation} 
(Here we are viewing $S^1$ as $\R/\Z$, i.e., as the circle of length 1.) 
It is closely related to the length functional, but it  has the advantage to select 
not only curves that extremize length but also their parameterization, 
namely one of constant speed. 

It is easy to prove that there exists $\ep > 0$, depending only on $N$ and $h$, 
such that if $\gamma$ is a critical point of $E$ and $E(\gamma) < \ep$ 
then $\gamma$ is constant. Indeed, $\ep = \frac12 \ell^2$ 
where $\ell$ is the length of the shortest closed geodesic in $(N^n,h)$
($\ep = +\infty$ if $(N^n, h)$ contains no closed geodesics). 

This is one of the simplest examples of a gap theorem in Geometric Analysis. 
More significant gap theorems, just to name a few, 
have been established in minimal surface theory, 
the study of Willmore surfaces, Yang-Mills theory and conformal geometry. 
A selection of such theorems can be found in 
\cite{sim}, \cite{mn}, \cite{mnoo}, \cite{cgz}, \cite{gks} and {\em many others} in the literature. 

Let now $(M^2,g)$ be a compact Riemannian surface without boundary. 
As a natural generalization of \eqref{eq:1d-energy} one can consider 
the \emph{Dirichlet energy} $E(u)$ defined for every $u\in W^{1,2}(M,N)$  by 
\begin{equation}
  E(u) = \frac{1}{2}\int_M |\nabla u|^2 \, dA_M = 
  \int_M e(u) \, dA_M,\label{dirichlet}
\end{equation} 
with $e(u)=\frac12 |\nabla u |^2$  the \emph{energy density} of $u$.

While \eqref{eq:1d-energy} enjoys nice compactness properties, 
this is not the case for \eqref{dirichlet}. For example, 
if $(M^2,g)$ is the standard 2-sphere then the Dirichlet energy is invariant 
under composition on the right with M\"obius maps. 
To deal with this issue, a {\em regularization} of the 
following form was considered in  \cite{sacks81}. 
For $\al>1$ and $u\in W^{1,2\al}(M,N)$,
\begin{equation}
  \Ea(u):=\frac{1}{2}\int_M (2+|\nabla u|^2)^\al \, dA_M  \label{energy}
\end{equation}
(indeed in \cite{sacks81} the integrand was $(1+|\nabla u|^2)^\al$ but 
it is convenient for our purposes to choose this equivalent variant). 
The functional $E_\al$ has the advantage of satisfying 
the Palais-Smale compactness condition, and hence critical points 
can be found in every homotopy class. 
Critical points of $E_\al$ are called $\al$-harmonic maps. 
They are smooth and, if $(N,h)$ is embedded isometrically into  $\R^k$ 
with second fundamental form $A$, they satisfy 
\begin{equation}
  \Delta u+A(u)(\nabla u,\nabla u) = -2(\al-1) (2+|\nabla u|^2)^{-1} 
  \langle \nabla^2 u, \nabla u \rangle \nabla u. \label{EL2}
\end{equation} 
Of course, a sequence of solutions with uniformly bounded $E_\al$-energy 
might still develop blow-ups as $\al$ tends to one, 
but in \cite{sacks81} it was shown that, after proper rescaling, 
it would form a finite number of {\em bubbles} 
(i.e., non-trivial harmonic maps from $S^2$ to $N$). 

Theorem 3.3 in \cite{sacks81} is the gap theorem for $\al$-harmonic maps that 
is analogous to the gap theorem for geodesics stated at the beginning of this Introduction. 

\begin{theorem} \label{su:gap} (\cite{sacks81}) 
There exists $\ep > 0$ and $\al_0 > 1$ such that if $E(u) < \ep, \ 1 < \al < \al_0$ 
and $u$ is a critical map of $E_{\al}$, then $u$ is constant. 
\end{theorem} 

From now on, $(M^2,g)$ and $(N^n, h)$ will both be 
$(S^2,g_{S^2})$, the standard 2-sphere isometrically embedded in $\R^3$. 
For $u \colon S^2 \rightarrow S^2$, its {\em Jacobian} is 
$J(u)=u \cdot e_1(u) \wedge e_2(u)$ and its {\em energy density} is 
$\frac12 |\nabla u|^2$, where $(e_1, e_2)$ stands for 
a local oriented orthonormal frame of $TS^2$. 
The {\em degree} of $u$ is then 
\begin{equation}
  \deg (u)=\frac{1}{4\pi} \int_{S^2} J(u) \, dA_{S^2}.\label{deg}
\end{equation}
In our previous paper (\cite{lmm}, Proposition $7.1$) we improved the gap theorem above for maps of degree zero.
\begin{proposition}\label{degzero}
Fix $\eta>0$.Then there exists $\overline{\al}-1>0$ small, $\overline{\al}$ depending only on $\eta$, such that if $1<\al \leqslant \overline{\al}$ and $u:S^2\to S^2$ is $\al$-harmonic, of degree zero and $E(u)\leqslant 8\pi-\eta$, then $u$ is constant.
\end{proposition}
While there exist $\alpha$-harmonic maps of degree zero 
whose energy converges from above to $8\pi$ as $\alpha \searrow 1$, see Theorem \ref{t:optimal} below, it is an interesting 
question to understand whether the convergence of the energies might be from below. 
We have preliminary numerical evidence that this somewhat unexpected phenomenon may, in fact, occur. 
\newline
\newline
If $u \in W^{1,2\al}(S^2,S^2)$ 
has degree one, we have 
\begin{align}
  8\pi &= \int_{S^2} (1+J(u)) \, dA_{S^2} \nonumber \\
  &\leqslant \int_{S^2} (1+e(u)) \, dA_{S^2} \label{est1} \\ \nonumber 
  &\leqslant \frac12 \, (2 \Ea(u))^{\frac{1}{\al}} 
  (4\pi)^{\frac{\al-1}{\al}}, 
\end{align}
which implies 
\begin{equation}
  \Ea(u)\geqslant 4^{\al} 2\pi \label{est2}. 
\end{equation}
We have equality if and only if $u$ is conformal 
with constant energy density equal to one, 
which happens only if $u$ is a rotation. 
In \cite{lmm} we proved the following result which is again an extension of the gap theorem 
\ref{su:gap} to this special setting. 

\begin{theorem} \label{t:main} (\cite{lmm})
There exists $\ep > 0$ and $\overline{\al}-1 > 0$  small 
such that the only critical points $\ua$ of $\Ea$ which satisfy 
$\Ea (\ua) \leqslant 4^{\al} 2\pi + \ep$ and $\al \leqslant \overline{\al}$ 
are the constant maps and the maps of the form $u^R(x)=Rx$ with $R \in O(3)$. 
\end{theorem} 

Since the only minimizers of $E$ among maps from $(S^2,g_{S^2})$ to itself of 
degree one are rational (see \cite{eelm}), it is natural to expect that the functions 
$u_\al$ as in Theorem \ref{t:main} are close in $W^{1,2}(S^2)$ to a rational map. 
Indeed, in \cite{lmm} this is proved in stronger norms and for specific 
rational maps (which are of M\"obius type, the degree being equal to $1$). 
A finite-dimensional reduction of the form in \cite{amma1}, together with refined 
asymptotic expansions then yields Theorem \ref{t:main}. It is somehow 
natural to expect the above result, since the $\al$-energy {\em breaks} 
the M\"obius symmetry and penalizes functions with large gradients. 

In this paper, exploiting some bubbling estimates of new type from \cite{lizhu}, 
we are able to quantitatively and optimally improve Theorem \ref{t:main} in order to obtain a result which resembles the one obtained in Proposition \ref{degzero}. 
This is our result. 
\begin{theorem} \label{t:new} 
Fix $\eta>0$. Then there exists $\overline{\al}-1 > 0$ 
so that the only critical points $\ua$ of $\Ea$ of degree $\pm 1$ which satisfy 
$\Ea (\ua) \leqslant 8^{\al} 2 \pi-\eta$ and $1<\al \leqslant \overline{\al}$ 
are maps of the form $u^R(x)=Rx$ with $R \in O(3)$. 
\end{theorem} 

In the final section of this paper we shall show that Proposition \ref{degzero} is optimal in the following sense. 
\begin{theorem} \label{t:optimal} 
For every $\ep > 0$, there exists $\al_0 > 1$ which depends only on $\ep$ such that, 
if $1 < \al < \al_0$ there exists an $\al$-harmonic map $\ua \colon S^2 \to S^2$ with 
$\deg(\ua) = 0$ and $6^{\al} 2 \pi \leqslant \Ea(\ua) < 6^{\al} 2 \pi + \ep$. 
\end{theorem} 

The $\al$-harmonic map in Theorem \ref{t:optimal} is produced by the 
methods of Section 8 in \cite{lmm} where we constructed an example of 
an $\al$-harmonic map $\ua$ of degree one with $\Ea(\ua) > 8^{\al} 2 \pi$. 
We believe that this degree 1 example is also optimal in a sense similar to that in 
Theorem \ref{t:optimal}, i.e., $\Ea(\ua)$ is close to $8^{\al} 2 \pi$ if $\al$ is close to 1. 

It would be interesting to understand whether 
one could replace the energy bounds in Theorem \ref{t:new} 
by Morse-index estimates as in \cite{micallef88}, \cite{uhlenbeck72}). 

\bigskip 
\noindent {\bf Acknowledgements} Tobias Lamm 
gratefully acknowledges financial support by the 
Deutsche Forschungsgemeinschaft (DFG) through 
CRC $1173$ and RTG $2229$. 
Andrea Malchiodi is supported by the project 
{\em Geometric problems with loss of compactness} 
from Scuola Normale Superiore and by 
MIUR Bando PRIN 2015 2015KB9WPT$_{001}$.  
He is also a member of GNAMPA as part of INdAM.  

\section{Proof of Theorem \ref{t:new}}

Without loss of generality, we assume that $u_\al$ is 
a critical point of $E_\al$ with $\deg(u_{\al})= 1$ and  
\[
E_{\al}(u_\al)\leqslant 8^{\al}2\pi-\eta.
\]
It then follows from H\"older's inequality that 
\begin{align}
  E(u_\al) &= \int_{S^2} (1+e(\ua)) \, dA_{S^2}-4\pi  \nonumber \\
  &\leqslant \frac12 \, (2 \Ea(\ua))^{\frac{1}{\al}} 
  (4\pi)^{\frac{\al-1}{\al}}-4\pi \nonumber \\
  &\leqslant 12\pi-\tilde{\eta} \label{boundE}
\end{align}
where $\tilde{\eta}>0$.

We claim that for $(\al-1)$ small enough all these critical points have to be 
of the form $u_{\al}(x)=Rx$ with $R\in SO(3)$. 

Assume, for a contradiction, that there exists a sequence $\al_k \searrow 1$ and 
a sequence $u_k:=u_{\al_k}$ of critical points of $E_k:=E_{\al_k}$ with 
$\deg(u_k)= 1$ and $E_k(u_k) \leqslant 8^{\al_k}2\pi-\eta$ 
which are not of the form $u_k=R_kx$ with $R_k\in SO(3)$. 
It follows from Theorem 1.1 in \cite{lizhu} that, 
for a subsequence (which we still denote by $u_k$), the energy identity
\[
12\pi\ >\lim_{k\to \infty} E(u_k) = 
\sum_{i=1}^m E(\omega^i)=4\pi\sum_{i=1}^m|\deg(\omega^i)|,
\]
where $m\in \N$ and the $\omega^i \colon S^2\to S^2$, $1\leqslant i \leqslant m$, 
are non-trivial harmonic maps, holds true. 
More precisely, the authors show that $\lim_{k \to \infty} r_k^{1-\alpha_k}=1$ 
(see equation (4.9) in \cite{lizhu}), where $r_k$ is 
an energy concentration radius on which a bubble forms. 
This is precisely the information which is needed in \cite{lamm10} 
(the entropy condition in this paper was only used in order to obtain 
the same limiting relation between $\al_k$ and the concentration radius) 
or \cite{liwang} in order to conclude the even stronger energy identity
\begin{align}
\lim_{k\to \infty} E_k(u_k) = 4\pi + \sum_{i=1}^m E(\omega^i) = 
4\pi(1+\sum_{i=1}^m |\deg(\omega^i)|). \label{energyidentity}
\end{align}
Note that  these results rely crucially on the assumption that 
the target manifold is a round sphere since it is known that 
the energy identity is not true in general for a sequence of $\al$-harmonic maps; 
see \cite{liwang2}.
Finally, it follows from Theorem 2 in \cite{duzkuw} that we have the decomposition
\[
1= \deg(u_k) =\sum_{i=1}^m \deg(\omega^i).
\]
Now we get from the above estimates that 
\[
\sum_{i=1}^m |\deg(\omega^i)| < 3 \quad\text{and}\quad 
\sum_{i=1}^m \deg(\omega^i)=1. 
\] 
Since $\deg(\omega^i) \neq 0$, the first inequality gives $m = 1$ or $2$ 
and the second equality then only allows $m=1$ and $\deg(\omega^1) = 1$. 
Looking back at \eqref{energyidentity} we see that, for every $\delta > 0$ 
there exists $k$ large enough so that 
\[
E_k(u_k) < 8\pi + \delta \leqslant 4^{\al_k} 2 \pi + \delta/2.
\]
Hence it follows from Theorem \ref{t:main} that we get a contradiction to the fact that 
$u_k$ is not a map of the form described in this Theorem.

\begin{remark}
Here we want to highlight that the above arguments for maps of degree one also yield a new proof of Proposition \ref{degzero}.

Namely, by following the same contradiction argument as above, 
we obtain a sequence $\al_k \searrow 1$ and 
a sequence of non-constant $\al_k$-harmonic maps $u_k$ with $\deg(u_k) =0$ 
and finitely many non-trivial harmonic maps $\omega^i \colon S^2\to S^2$, 
$1 \leqslant i \leqslant m$, so that 
\begin{gather*}
 8 \pi > \lim_{k\to \infty} E(u_k)= 4\pi \sum_{i=1}^m |\deg(\omega^i)|,\\
\lim_{k\to \infty} E_k(u_k) = 4\pi(1+\sum_{i=1}^m|\deg(\omega^i)|) 
\quad\text{and}\quad \deg(u_k)=0=\sum_{i=1}^m \deg(\omega^i).
\end{gather*}
It follows that $m=1$, $\deg(\omega^1)=0$ which contradicts the fact that $\omega^i$ is supposed to be non-constant.
\end{remark}

\section{Proof of Theorem \ref{t:optimal}} 

\subsection{Construction of a non-constant $\al$-harmonic map of degree zero} 
As in Section $8$ in our paper \cite{lmm}, 
we construct the desired map as a minimizer of the $\al$-energy 
in a certain class of co-rotationally symmetric maps. 
More precisely, we consider the parameterisation
\[ 
(r,\theta) \mapsto (\sin r \, \cos \theta, \ \sin r \, \sin \theta, \ \cos r)
\] 
of $S^2$, where $n \in \N, \ r \in [n\pi, (n+1)\pi]$ and $\theta \in [0,2\pi]$. 
We are then interested in maps $u_f \colon S^2\to S^2$ such that 
\[ 
(r,\theta) \mapsto (\sin (f(r)) \cos \theta, \ \sin (f(r)) \sin \theta, \ \cos (f(r)))
\]
with
\[
f \colon [0,\pi] \to \R, \ f(0) = 0, \ f(\pi) = 2 \pi. 
\] 
Such a map $u_f$ has degree zero and for 
\[ 
X := \{ f \colon [0,\pi] \to \R : u_f \in W^{1,2\al}(S^2, \R^3), \ \ f(0) = 0, \ f(\pi) = 2\pi\}
\] 
the infimum $\Lambda := \inf_{f\in X}I(f)$, where 
\[ 
I(f) := \Ea(u_f) = \pi \int_0^{\pi} \left(2 + (f')^2 + \frac{(\sin f)^2}{(\sin r)^2}\right)^{\al} \sin r \, dr 
\] 
is attained by a map $f^*\in X$ as was shown in \cite{lmm}. 
(Note that the argument given there trivially extends to the case considered here.) 
Moreover, setting $u^* := u_{f^*}$, one estimates 
\begin{align*} 
\Ea(u^*) &= \pi \int_0^{\pi} 
\left(2 + (f^*\mbox{}')^2 + \frac{(\sin f^*)^2}{(\sin r)^2}\right)^{\al} 
\sin r \, dr \\[2\jot] 
&\geqslant \pi \left(\int_0^{\pi} \left(2 + (f^*\mbox{}')^2 + \frac{(\sin f^*)^2}{(\sin r)^2}\right) 
\sin r \, dr \right)^{\al}  \left(\int_0^{\pi} \sin r \, dr \right)^{1 - \al} \\[2\jot] 
&\geqslant 2^{1 - \al} \pi \left(\int_0^{\pi} 
\left(2 \sin r + 2 |f^*\mbox{}'(\sin f^*)| \right) \, dr \right)^{\al}. 
\end{align*} 
Now there exists $r_1\in (0,\pi)$ such that $f^* (r_1)=\pi$ and hence
\begin{align*} 
\int_0^{\pi} |f^*\mbox{}'(\sin f^*)| \, dr &\geqslant \int_0^{r_1} f^*\mbox{}'(\sin f^*) \, dr - 
\int_{r_1}^{\pi} f^*\mbox{}' (\sin f^*) \, dr  \\ 
&= \left.-\cos f^*(r)\right|_0^{r_1} + \left.\cos f^*(r)\right|_{r_1}^{\pi}  \\ 
&= 4
\end{align*}
which implies $\Ea(u^*) \geqslant  6^\al 2\pi$. 
\begin{remark}
If we minimize $I(f)$ among maps $f \colon [0,\pi] \to [0,2k\pi]$ 
then we get an $\al$-harmonic map of degree 0 and 
alpha-energy at least $(4k+2)^{\al} 2 \pi$
and if we minimize $I(f)$ among maps $f \colon [0,\pi] \to [0,(2k+1)\pi]$ 
then we get an $\al$-harmonic map of degree $1$ and 
alpha-energy at least $(4k+4)^{\al} 2 \pi$. 
\end{remark}

In what follows we shall rename the map $u^*$ constructed above as $\ua$, 
to indicate its dependence on $\al$. 

\subsection{An upper bound for $\Ea(\ua)$.} 
\begin{proposition}\label{p:ub} There exists   $C > 0$ universal such that, 
for $1 < \al < \frac54$, 
\begin{equation} \label{deg0:ub} 
\Ea(\ua) < 6^{\al} 2 \pi + C (\al - 1)^{1/2}. 
\end{equation} 
\end{proposition} 

\begin{proof} For $\l > 1$, let $f(r) := 2 \arctan(\l \tan r), \ r \in [0,\pi]$ 
with $\arctan$ also taking values in $[0,\pi]$. 
Then as $r$ increases from 0 to $\pi/2, \ f(r)$ increases from 0 to $\pi$, 
i.e., $u_f$ maps the upper hemisphere to the full sphere with 
the equator being all mapped to the South Pole $(0,0,-1)$. 
The requirement that $\arctan$ takes values between 0 and $\pi$ forces 
$f(r)$ to increase from $\pi$ to $2\pi$ as $r$ increases from $\pi/2$ to $\pi$, 
i.e., $u_f$ maps the lower hemisphere to the full sphere again 
but with the opposite orientation (which confirms that $u_f$ has degree zero). 
Furthermore, $u_f \in X$. 

In \S 8.1 of \cite{lmm}, we showed that the energy density $e(u_f)$ is given by 
\[ 
e(u_f) = \frac12 \left( (f')^2 + \frac{(\sin f)^2}{(\sin r)^2}\right). 
\] 
So we calculate 
\[ 
f'(r) = \frac{2 \l}{\cos^2 r + \l^2 \sin^2 r}, \qquad 
\sin(f(r)) = \frac{2 \l \tan r}{1 + \l^2 \tan^2 r} = (\sin r) (\cos r) f'(r). 
\] 
Therefore, 
\begin{equation} \label{eq:edensity} 
e(u_f) = 2 \l^2 \frac{1 + \cos^2 r}{(\l^2 - (\l^2 -1)\cos^2 r)^2} 
\end{equation} 
and 
\begin{equation}\label{eq:Ea1} 
\Ea(u_f) = 2^{\al} \pi \int_0^{\pi}(1 + e(u_f))^{\al} (\sin r) \,dr 
= 2^{\al + 1}\pi \int_0^{\pi/2}(1 + e(u_f))^{\al} (\sin r) \,dr,  
\end{equation} 
where we have used the even symmetry of $\cos^2 r$ and $\sin r$ about $\pi/2$. 
We shall use the inequality 
\begin{equation}\label{eq:alineq} 
(1 + x)^{\al} \leqslant 1 + \al x + (\al -1)x^2, \quad \al \in [1,2], \ x \geqslant 0 
\end{equation} 
and therefore, we need to estimate the integrals 
\begin{align}  
I_1&:= 
\int_0^{\pi/2} \frac{1 + \cos^2 r}{(\l^2 - (\l^2 -1)\cos^2 r)^2} \, (\sin r) \, dr 
\label{eq:I1} 
\\ 
I_2&:=
\int_0^{\pi/2} \frac{(1 + \cos^2 r)^2}{(\l^2 - (\l^2 -1)\cos^2 r)^4} \, (\sin r) \, dr . 
\label{eq:I2} 
\end{align} 
We make the change of variables $t = \cos r$ and set $a^2 := 1 - \l^{-2}$. 
Then \eqref{eq:I1} and \eqref{eq:I2} become 
\begin{align}  
I_1&:= 
\l^{-4} \int_0^1 \frac{1 + t^2}{(1 - a^2 t^2)^2} \, dt 
\label{eq:I11} 
\\ 
I_2&:=
\l^{-8} \int_0^1 \frac{(1 + t^2)^2}{(1 - a^2 t^2)^4} \, dt . 
\label{eq:I22} 
\end{align} 
To estimate \eqref{eq:I11} we use 
\[ 
\frac{1 + t^2}{(1 - a^2 t^2)^2} < \frac{1}{2a^2} \left( 
\frac{1}{(1-at)^2} + \frac{1}{(1+at)^2} \right) 
\] 
and note that 
\[ 
\int_0^1 \left( \frac{1}{(1-at)^2} + \frac{1}{(1+at)^2} \right) \, dt 
= \frac{2}{1-a^2} =2\l^2.
\] 
Therefore, 
\begin{equation} \label{eq:E1} 
\int_0^{\pi /2} e(u_f) \, (\sin r) \, dr
= 2 \l^2 I_1 < \frac{2}{a^2} = \frac{2\l^2}{\l^2 - 1}.
\end{equation} 
To estimate \eqref{eq:I22} we use 
\[ 
\frac{(1 + t^2)^2}{(1 - a^2 t^2)^4} < \frac{1}{a^4(1-at)^4} 
\] 
and note that 
\[ 
\frac{1}{a^4} \int_0^1 \frac{1}{(1-at)^4} \, dt = 
\frac{1}{3a^5} \left(\frac{1}{(1-a)^3} - 1 \right) < \frac{1}{a^4(1-a)^3}. 
\] 
Now 
\[ 
\frac{1}{a^4(1-a)^3} = \frac{(1+a)^3}{a^4(1-a^2)^3} < 8 \l^6 a^{-4} 
\] 
and therefore, if $\l^2 > 2$ 
so that $a^2 > \frac12$ we have that 
$I_2 < 32 \l^{-2}$. It follows that 
\begin{equation} \label{eq:E2} 
\int_0^{\pi /2} \big( e(u_f)\big)^2 \, (\sin r) \, dr = 4 \l^4 I_2 < 128 \l^2. 
\end{equation} 
Using \eqref{eq:alineq} in \eqref{eq:Ea1} we get 
\[ 
\Ea(u_f) \leqslant 2^{\al + 1} \pi \left( 1 + \al \int_0^{\pi /2} e(u_f) \, (\sin r) \, dr + 
(\al - 1) \int_0^{\pi /2} \big( e(u_f)\big)^2 \, (\sin r) \, dr \right) 
\] 
from which it follows that 
\[ 
\Ea(u_f) < 2^{\al + 1} \pi \left( 1 + \frac{2 \al \l^2}{\l^2 - 1} + 128 (\al - 1) \l^2\right). 
\] 
Still assuming that $\l^2 > 2$ we have 
$\frac{\l^2}{\l^2 - 1} < 1 + 2 \l^{-2}$ and therefore, 
\[ 
\Ea(u_f) < 2^{\al + 1} \pi \left( 3 + \frac{4\al}{\l^2} + 2(\al - 1) + 128 (\al - 1) \l^2\right). 
\] 
By choosing $\l = (\al - 1)^{-(1/4)}$ (with $\al \in (1,\frac54)$ so that $\l^2 > 2$) 
we obtain 
\[ 
\Ea(u_f) < 2^{\al + 1} \pi \big( 3 + C(\al - 1)^{1/2}\big) 
< 6^{\al} 2 \pi + C(\al - 1)^{1/2}. 
\] 
Finally, since $\ua$ minimizes $I$ among maps in $X$, we have that 
\[ 
\Ea(\ua) < \Ea(u_f) < 6^{\al} 2 \pi + C(\al - 1)^{1/2}, 
\] 
namely, \eqref{deg0:ub}. 
\end{proof} 
\subsection{Completion of Proof of Theorem \ref{t:optimal}.} 
\begin{proof} 
Given $\ep > 0$ choose $\frac54 > \al_0 > 1$ so that $C(\al_0 - 1)^{1/2} < \ep$. 
Then, by Theorem \ref{t:new} and Proposition \ref{p:ub}, for $1 < \al < \al_0$ we have 
\[ 
6^{\al} 2 \pi \leqslant \Ea(\ua) < 6^{\al} 2 \pi + C(\al - 1)^{1/2} < 6^{\al} 2 \pi + \ep, 
\] 
which concludes the proof. 
\end{proof} 

\begin{remark} The map $u_f$ constructed above with $f(r) = 2 \arctan(\l \tan r)$ 
can also be written as 
\[ 
u_{\l}(z) = \frac{2\l z}{1-|z|^2}, \quad z \in \C \cup \{\infty \}, 
\] 
where $S^2$ is identified with $\C \cup \{\infty\}$ 
via stereographic projection from the South Pole. 
The calculations seemed to us somewhat more messy in this representation of $S^2$. 
However, this representation suggests the following example for 
proving the optimality of the degree 1 example referred to 
just after Theorem \ref{t:optimal}: 
\[ 
u_{\ep}(z) := \frac{z(1 - \ep^2|z|^2)}{(\ep^2-|z|^2)}, \quad 0 < \ep \ll 1, 
\text{ depending on $\al$}. 
\] 
As $\ep \downarrow 0$, $u_{\ep}$ develops a holomorphic bubble at $0$ 
and also at $\infty$ and, on compact subsets of $0 < |z| < \infty$, 
$u_{\ep}(z)$ converges smoothly to $\frac{1}{\bar{z}}$. 
It is straightforward to check that the corresponding 
$f_{\ep}\colon [0,\pi] \to [0,3\pi]$ is given by 
\[ 
f_{\ep}(r) := 2 \arctan \left(
\frac{(\tan(r/2))(1-\ep^2 \tan^2(r/2))}{\ep^2 - \tan^2(r/2)}
\right).
\] 
(The values taken by $\arctan$ in the vicinity of 
$\tan(r/2) = \ep$ and $\tan(r/2) = 1/\ep$ 
have to be chosen appropriately so as to make $f_{\ep}$ continuous.) 
This example shows how the weak limit of degree 1 $\al$-harmonic maps 
may have degree $-1$. 
\end{remark}

\end{document}